\documentclass{amsart}

\scrollmode
\nonstopmode
\newtheorem{thm}{Theorem}[section]

\newtheorem{lem}[thm]{Lemma}
\newtheorem{prop}[thm]{Proposition}

\newtheorem{defi}[thm]{Definition}

\newtheorem{rema}[thm]{Remark}
\numberwithin{equation}{section}

\def\bbR {\mathbb{R}}

\def\calO {\mathcal{O}}

\def\calE {\mathcal{E}}

\let\lto\longrightarrow

\def\pv {\;\string;\;}

\def\barC {\mkern 2mu \overline{\mkern -2mu C}}

\def\tttv {|\mkern -2mu \|} 

\def\puce {\smallskip
\indent$\scriptstyle\triangleright$\enspace\ignorespaces}

\let\rm\normalfont

\begin{document}

\title
[Finite reflection groups: Invariant functions]
{Finite reflection groups:\\
Invariant functions and\\
functions of the invariants\\ 
in finite class of differentiability}

\author{ G\'erard P. BARBAN\c CON}

\address{University of Texas at Austin}

\email{gbarbanson@yahoo.com }




\begin{abstract}
Let $W$ be a finite reflection group. A $W$-invariant function of class~$C^{\infty}$ may be expressed as a functions of class $C^{\infty}$ of the basic invariants. In finite class of differentiability, the situation is not this simple. Let~$h$ be the  greatest Coxeter number of the irreducible components of $W$  and $P$ be~the Chevalley mapping, if $f$ is an invariant function of class $C^{hr}$, and $F$ is the function of invariants associated by $f=F\circ P$, then $F$ is of class $C^r$. However if~$F$ is of class $C^r$, in general $f=F\circ P$ is of class $C^r$ and not of class $C^{hr}$. Here we determine the space of $W$-invariant functions that may be written as functions of class $C^r$ of the polynomial invariants and the subspace of functions $F$ of class $C^r$ of the invariants such that the invariant function $f=F\circ P$ is of class $C^{hr}$. 
\end{abstract}

\maketitle

\section{Introduction}

Let $W$ be a finite reflection group acting orthogonally on $\bbR^n $. The algebra of invariant polynomials is generated by $n$ algebraically independent $W$-invariant homogeneous polynomials and the degrees of these polynomials are uniquely determined [3], [4]. Let $p_1,\ldots,p_n$ be a set of basic invariants, we define the ``Chevalley'' mapping
\[
P:\bbR^n \lto \bbR^n, \quad  
P(x)=\big(p_1(x), ..., p_n(x)\big).
\]
Glaeser's theorem [7] shows that $W$-invariant functions of class $C^{\infty}$ may be expressed as functions of class $C^{\infty}$ of the basic invariants, and that the subalgebra $P^*(C^{\infty}(\bbR^n))$ of composite mappings of the form~${F\circ P}$ with $F$ of class~$C^{\infty}$, is closed in $C^{\infty}(\bbR^n)$. 

In finite class of differentiability, the situation is not this simple. Let~$h$ be the highest degree of the coordinate polynomials in $P$, equal to the greatest Coxeter number of the irreducible components of $W$. We have:   

\begin{prop}  [see {[1]}]\label{Yrefprop}
There exists a linear and continuous mapping
$$
C^{hr}(\bbR^n)^W \ni f \lto F \in C^r (\bbR^n)
$$
such that $f =F\circ P$.
\end{prop}
 A~general counter-example shows that this result is the best possible.  

However if $F$ is of class $C^r$, in general $f=F\circ P$ is of class~$C^r$ and not of class~$C^{hr}$. So the following questions naturally arise: 

1) What is the space of $W$-invariant functions that may be written as functions of class $C^r$ of the polynomial invariants, and 

2) What is the subspace of functions $F$ of class $C^r$ of the invariants such that the invariant function $f=F\circ P$ is of class $C^{hr}$?  To answer this second question we will complete the results given in [1] and determine the subspace of $\calE^r(P(\bbR^n))$ isomorphic to $C^{hr}(\bbR^n)^W$.

\section{The Chevalley mapping}

When the reflection group $W$ is reducible, it is the direct product of its irreducible components, say 
$$
W=W^0\times W^1\times \cdots \times W^s
$$ 
and we may write $\bbR^n=\bbR^{n_0} \oplus \bbR^{n_1}\oplus \cdots\oplus \bbR^{n_s}$   as an orthogonal direct sum. The first component $W^0$ is the identity on the $W$-invariant subspace~$\bbR^{n_0}$, and $W^i$ is an irreducible finite Coxeter group acting on~$\bbR^{n_i}$
\hbox{for $i=1,\ldots,s$}. We  choose coordinates that fit with this orthogonal direct sum: if $w=w_0w_1 \cdots w_s \in W$ with $w_i\in W^i, 0\leq i\leq s$, we have for  all $x\in \bbR^n$
$$
w(x)=w(x^0,x^1,\ldots,x^s)=\big(x^0, w_1(x^1), \ldots, w_s(x^s)\big).
$$
The direct product of the identity $P^0$ acting on $\bbR^{n_0}$ and Chevalley mappings $P^i$ associated with $W^i$ acting on $\bbR^{n_i}$, $1\leq i\leq s$, is a Chevalley map associated with~$W$ acting on $\bbR^{n}$.

For an irreducible $W$ (or for an irreducible component) we   may assume    that the degrees of the polynomials $p_1, \ldots, p_n$ are in increasing order 
$$
k_1=2\leq \cdots \leq k_n=h , 
$$
where $h$ is the Coxeter number of~$W$. In the reducible case, if there is an invariant subspace $\bbR^{n_0}$, the corresponding $p_i$s say $P^0_j, 1\leq j\leq n_0$ are of degree $k_0=1$. Then for each $i=1,\ldots,s$ the $P^i_j$ are of degree~$k^i_j$, $1\leq j\leq n_i$. The $k^i_1$ are equal to~$2$, and~$k^i_{n_i}$ is the Coxeter number of $W^i$. We denote with $h$ the maximal Coxeter number  (or highest degree of the coordinate polynomials) of the irreducible components.

Let $\mathcal R$ be the set of reflections different from identity in $W$. The number of these reflections is 
$$
\mathcal R^{\#}=d=\sum_{i=1}^n (k_i-1).
$$
For each $\tau\in \mathcal R$, we consider a linear form $\lambda_{\tau}$, the kernel of which is the reflecting hyperplane 
$$
H_{\tau}=\big\{x\in \bbR^n \pv \tau(x)=x\big\}.
$$
The jacobian matrix $J_P$ of $P$, is block diagonal. 
Let $I_{n_0}, J_{P^1}, \ldots, J_{P^{n_s}}$ be the diagonal blocks. The $n_0 \times n_0$ minor in the upper left corner is~$1$ but the determinants $|J_{P^1}|, \ldots, |J_{P^{n_s}}|$, all vanish on $\bbR^{n_0}$. The Jacobian determinant of $P$ is 
$$
|J_P|= c\,\prod_{\tau \in \mathcal R} \lambda_{\tau}
$$
for some constant $c\neq 0$. The critical set is the union of the $H_{\tau}$ when $\tau$ runs through $\mathcal R$. The invariant subspace is the intersection $\bigcap_{\tau \in \mathcal R}H_{\tau}$ of the reflection hyperplanes.

A {\it Weyl chamber} $C$ is a connected component of the regular set. The other connected components are obtained by the action of $W$ and the regular set is $\bigcup_{w\in W} w(C)$. There is a stratification of $\bbR^n$ by the regular~set, the reflecting hyperplanes $H_{\tau}$ and their intersections. The~mapping $P$ is neither injective nor surjective but it induces an analytic diffeomorphism of $C$ onto the interior of $P(\bbR^n)$ and an homeomorphism that carries the stratification from the fundamental domain~$\barC  $ onto~$P(\bbR^n)$.
Finally let us recall that the Chevalley mapping is proper and separates the orbits.

\section{Whitney functions and $r$-regular jets of order $m$}

A general study of Whitney functions may be found in [9], the notations of which will be used freely.

A {\it jet of order} $m\in \mathbb N$, on a locally closed set $E\subset \bbR^n$ is a collection 
$$
A= (a_k)_{k\in \mathbb N^n, |k|\leq m}
$$
of real valued functions  $a_k$ continuous on $E$. At each point $x\in E$ the jet~$A$ determines a polynomial of degree $m$, $A_x(X)$ and we will sometimes speak of continuous polynomial fields instead of jets. 
As a function, $A_x(X)$ acts on vectors $x'-x$ tangent to $\bbR^n$ and we write
\[ 
A_x:x'\longmapsto A_x(x')=\sum_{\substack{k\in \mathbb N^n\\ |k|\leq m}} \frac{1}{k!} a_k(x) (x'-x)^k  .
\]
The space $J^m(E)$ of jets of order $m$ on $E$ is naturally provided with the Fr\'echet topology induced by the family of semi-norms 
\[
|A|_{K_s}^m = \sup_{\substack{x\in K_s\\ |k|\leq m}} \big|a_k(x)\big|, 
\] 
 where $K_s$ runs through an exhaustive countable collection of compact sets of $E$.

By formal derivation of $A$ of order $q\in \mathbb N^n, |q|\leq m$, we get jets of the form $(a_{q+k})_{|k|\leq m-|q|}$, inducing polynomials
\[ 
(D^qA)_x(x') =a_q(x)+\sum_{\substack{k>q\\ |k|\leq m}} \frac{1}{(k-q)!}a_k(x)(x'-x)^{k-q}.
\]

For $0\leq |q|\leq r\leq m$, we put
\[ 
(R_xA)^q(x')= (D^qA)_{x'}(x')-(D^qA)_{x}(x').
\]

\begin{defi}\rm
Let $A$ be a jet of order $m$ on $E$. For $r\leq m$, we say that $A$ is $r$-{\it regular on} $E$, if   for all compact set $K\subset E$, for $(x,x')\in K^2$, and for all $q\in \mathbb N^n$ with $|q|\leq  r$, it satisfies the {\it Whitney conditions}
\[ (W_q^r)\hskip 2cm
(R_xA)^q(x')=o\big(|x'-x|^{r-|q|}\big),  
\]
\end{defi}

\begin{rema}\rm
   If $m>r$ there is no need to consider the truncated field $A^r= (a_k)_{k\in \mathbb N^n, |k|\leq r}$ instead of $A$ in the conditions $(W_q^r)$. Actually $(R_xA^r)^q(x')$ and $(R_xA)^q(x')$ differ by a sum of terms $[\frac{a_k(x)}{(k-q)!}](x'-x)^{k-q}$, with $a_k$ uniformly continuous on $K$ and $|k|-|q|>r-|q|$.
\end{rema}

The space of $r$-regular jets of order $m$ on $E$ is naturally provided with the Fr\'echet topology defined by the family of semi-norms
\[\|A\|^{r,m}_{K_s}=|A|_{K_s}^m
+\sup_{\substack{ (x,x')\in K^2_s \\x\neq x',\, |k|\leq r }}
\Big(\frac{|(R_xA)^k(x')|}{|x-x'|^{r-|k|}}\Big)\]
when $K_s, s\in \mathbb N$ runs through an exhaustive family of compact sets in~$E$. This Fr\'echet space will be denoted by $\mathcal E^{r,m}(E)$.

\medskip

If $r=m, \mathcal E^r(E)$ is the space of Whitney fields of order $r$ or Whitney functions of class $C^r$ on $E$. When $E$ is locally closed we have:

\begin{thm} [see {[11]}]\label{thmref}
The restriction mapping of the space $\mathcal E^r(\bbR^n)$ to the space 
$\mathcal E^r(E)$ is surjective. There is a linear section, continuous when the spaces are provided with their natural Fr\'echet topologies.\end{thm}

When $A\in \mathcal E^r(E)$ and $|q|\leq r$, the formal derivation $D^qA$ and the regular derivative $\partial^{|q|}A/\partial x^q$ define the same polynomial 
\[
(D^qA)_x(x') = \Big( \frac{\partial^{|q|}A}{\partial x^q}\Big)_{\!x}(x')
\]
and they may be identified.

\smallskip

In general the semi-norms $\|.\|^r_{K_s}$ and $|.|^r_{K_s}$ are not equivalent on $\calE^r(E)$.  However if $E$ is an open set $\calO$, the semi-norms are equivalent. In this case $\calE^r(\calO)$ is the space of fields of Taylor polynomials of functions in~$C^r(\calO)$ and the two spaces $\calE^r(\calO)$ and~$C^r(\calO)$ may be identified.

\begin{defi}  [see {[10]}]\rm
We say that a compact set $K \subset \bbR^n$ connected by rectifiable arcs, is {\it Whitney $1$-regular} if the geodesic distance in K is equivalent to the Euclidian distance.
That is,  there is a constant $k_K>0$ such that for all $(x,x')\in K^2$, there is a rectifiable arc from $x$ to $x'$ in $K$ with length $\ell(x,x')\leq k_K|x-x'|$. 
\end{defi}

\begin{prop}  [see {[10]}]\label{Yrefprop}
If the compact set $K$ is $1$-regular, the norms $\|.\|_K^r$ and~$|.|_K^r$ are equivalent on $\mathcal E^r(K)$.
\end{prop}

\begin{prop}  [see {[2]}]\label{Yrefprop}
The image by the Chevalley mapping $P$ of any closed ball centered at the origin is Whitney $1$-regular.
\end{prop}

This entails that any compact in $P(\bbR^n)$ is Whitney 1-regular.
Hence on $\calE^r(P(\bbR^n))$ the families of semi-norms $\|.\|_K^r$ and $|.|_K^r$ are equivalent.

\section{Functions of class $\mathcal{C}^r$ of the invariants and invariant functions} 

We consider the homomorphism induced by $P$:
\[
P^*: \calE^r\big(P(\bbR^n)\big) \longrightarrow \calE^r(\bbR^n),\quad \;P^*(F)=F\circ P .
\]
For any $(a,x)\in \bbR^n\times \bbR^n $, by the Taylor formula for $F$ between $P(a)$ and $P(x)$, either by using an extension of $F$ to $\calE^r(\bbR^n)$ given by Whitney extension theorem or by using a Taylor integral remainder along an integrable path satisfying the inequality given by 
the $1$-regularity property of $P(\bbR^n)$, we have:
\[
F\big[P(x)\big]=F\big[P(a)\big]
+\sum_{1\leq |\beta|\leq r} \frac{1}{\beta !} D^{\beta}
F\big[P(a)\big]\big(P(x)-P(a)\big)^{\beta}+o\big(|P(x)-P(a)|^r\big).
\]

Expanding $P(x)-P(a)$ by the polynomial Taylor formula, we get a polynomial in $x-a$ of degree $h$. Hence for $f=F\circ P$,
\[
f(x)=f(a)
+
\sum_{1\leq |\alpha|\leq hr} \frac{1}{\alpha !} f_{\alpha}(a) (x-a)^{\alpha} +o\big(|P(x)-P(a)|^r\big).
\]
On a compact $K\subset \bbR^n$ containing $[a,x]$, there exists a constant $C_K$ such that $|P(x)-P(a)|^r\leq C_K |x-a|^r $. 

\smallskip

When $|\alpha|\leq r$ the $f_{\alpha}$ are continuous derivatives. 
When $r<|\alpha|\leq hr$ the $f_{\alpha}$ are obtained by the composition process with $|\beta|\leq r$ and they are continuous, so $f$ belongs to $\mathcal E^{r,hr}(\bbR^n)$.    

\begin{lem}
Let $F$ be of class $C^r$ and $f=F\circ P$. For $0\leq |\beta|\leq r,\, F_{\beta}\circ P(x)$ is a linear combination of some $f_{\alpha}(x),\, 0\leq |\alpha|\leq hr$. 
\end{lem}

\begin{proof}
We have $f_0(x)=F_0(P(x))$.
By induction, let us first assume $r=1$. We identify
\[
f_x(x') = f_0(x)
+
\sum_{1\leq|\alpha| \leq h}
\frac{1}{\alpha  !} f_\alpha (x) (x_1'-x_1)^{\alpha_1} \cdots (x_\ell'-x_\ell)^{\alpha_\ell}
\]
with 
\[
f_x(x')=F_0\big(P(x)\big)
+\sum_1^\ell F_i\circ P(x) 
\Big(
\sum_{|\alpha|=1}^{k_i} 
\frac{1}{\alpha !} \frac{\partial^{|\alpha|}p_i}{\partial x^\alpha} (x) (x'-x)^{\alpha}
\Big),
\]
where $F_i$ stands for $F_\beta, \beta_i=1, \beta_j=0$ if $i \neq j $.

\smallskip

For $|\alpha|=k_i$, we get
\[
f_{\alpha} 
=
F_i\circ P \; 
\frac{\partial^{k_i}p_i}{\partial x^{\alpha}} 
+
\sum_{s>i} F_s\circ P\; \frac{\partial^{k_i}p_s}{\partial x^{\alpha}}\cdot
\]
In particular when $|\alpha|=h$, we have  $f_{\alpha}=F_n\circ P\; ({\partial^hp_n}/{\partial x^{\alpha})}$ and since~$p_n$ is of degree $h$ there is a $\partial^h p_n/\partial x^{\alpha}$ which is not $0$. 

\smallskip

Hence the result for $F_n$. Solving the equations in succession gives the result for the $F_i\circ P$, $i=1\ldots n$.

\smallskip

For more explicit computations, observe that if $W$ is reducible, it would be sufficient to study each irreducible component $W_i$ in each~$\bbR^{n_i}$ and gather the results at the end.
For an irreducible component we may use the polynomial invariants given in [8]. Disregarding $D_n$ for a while, for the other groups the $k_i$ are distinct and there is an invariant set of linear forms $\{L_1, \ldots, L_v\}$ such that their symmetric functions $\sum_{j=1}^v L_j^k$ are $W$-invariant and we may take $p_i=\sum_{j=1}^v [L_j(x)]^{k_i}$ with the~$k_i$s determined in [4]. At least one of the $L_j$ contains a monomial in $X_1$, bringing in $p_i(X)$ a monomial in $X_1^{k_i}$ that cannot be canceled since the~$k_i$ are all even, with two exceptions: $A_n$ and $I_{2p}$.

\puce
 For $I_{2p}$ we may choose 
 $$
 p_1(X)=X_1^2+X_2^2
 \quad \hbox{and} \quad
 p_2(X)= \sum_{i=1}^p (X_1 \cos 2j\theta+X_2\sin 2j\theta)^p
 $$ 
 in which the coefficient of $X_1^p$ is $\sum_{j=1}^P(\cos 2j\theta)^p \neq 0$.

\puce
For $A_n$ we may take $L_i=X_i, i=1,\ldots,n+1$ and there is no possible cancelation either. 

\smallskip

Finally for $D_n$ if we choose as basic invariant polynomials 
$$
p_j=\sum_{i=1}^n x_i^{2j} \enspace  (j=1, \ldots,n-1)
 \quad \hbox{and} \quad
 q(x)=x_1x_2\cdots x_n , 
 $$
 we may use the above method when $1\leq j\leq n-1$, and consider $\partial^n q/\partial x_1\cdots \partial x_n=1$ to get the continuity of $(\partial F/\partial q)\circ P$.

\smallskip

Anyway, one has  $\partial^{k_n}p_n/\partial x_1^{k_n}=k_n!c_n$ for some constant $c_n\neq 0$, while for $j<n,\; \partial^{k_n}p_j/\partial x_1^{k_n}=0$, since the greatest exponent of $x_1$ in $p_j(x)$ \hbox{is $k_j<k_n$}. The identification shows that 
$$
c_n\,F_n\circ P\,(x)= \frac 1  {k_n!}  f_{k_n,0,\ldots,0} (x)    \quad \hbox{with \ $c_n\neq 0$.}
$$

Assuming that when $s>i$, the $F_s\circ P$ are linear combinations of the~$f_{\alpha}$, 
$|\alpha|\leq h$, since $p_i(x)$ contains a monomial in $x_1^{k_i}$, we have 
$$
\frac{\partial^{k_i}p_i}{\partial x_1^{k_i}}=k_i! \, c_i
$$
for some constant $c_i\neq 0$, and $ \partial^{k_i}p_j/\partial x_1^{k_i}=0$ for $j<i$. The identification now gives
\[ 
f_{k_i,0,\ldots,0} 
=
c_i \, k_i! \, F_i\circ P
+
\sum_{s>i}F_s\circ P \frac{\partial^{k_i}p_s}{\partial x_1^{k_i}} \cdot
\]   
With the induction assumption we have the result for $F_i\circ P$ and by decreasing induction for all the $F_j\circ P, j=1,\ldots, n$.

Let us assume that the statement of the lemma is true when $r\leq k$. When $|\beta|=k, F_{\beta}\circ P$ is a linear combination of some $f_{\alpha}$s with $|\alpha|\leq hk$. By using the basis step of the induction for the function $g = G\circ P$ \hbox{with $G=F_{\beta}$}, we get that the $G_i\circ P$ are linear combinations of $g_{\alpha}$, $|\alpha|\leq h$. The induction  assumption for~$g$ shows that for $|\gamma|=k+1$, $F_{\gamma}\circ P$ is a linear combination of $f_{\alpha}$, $|\alpha|\leq h(k+1)$. This achieves the induction and the proof of the lemma.
\end{proof}

 From the lemma we get that for any compact $K\subset \bbR^n$ there is a constant $C_K$ depending only on $K$ and $W$ such that 
 $$
 |F|^r_{P(K)}\leq C_K |f|^{hr}_K  . 
 $$
 Any compact in $P(\bbR^n) $ being Whitney $1$-regular, there is a constant $C_K^1$ such that $\|F\|^r_{P(K)}\leq C_K^1 |f|^{hr}_K$, and of course  $\|F\|^r_{P(K)}\leq C_K^1 \|f\|^{r,hr}_K$.

The linear injective mapping $P^*$ is surjective on its image $\mathcal J$. 
This image
$\mathcal J$ is the space of invariant functions associated with fonctions of class $C^r$ of the invariants. By the above inequality, the inverse 
$$
(P^*)^{(-1)}:\mathcal J \longrightarrow \calE^r \big( P(\bbR^n)\big)
$$
is continuous. Hence by the Banach theorem it is an isomorphism of Fr\'echet spaces. The image $\mathcal J \subset{\calE^{r,hr}(\bbR^n)}$ of $P^*$ is closed in both $\calE^r(\bbR^n)$ and $\calE^{r,hr}(\bbR^n)$. 

The map
$$
P^*:\calE^r\big(P(\bbR^n)\big)\longrightarrow \mathcal J
$$ 
is also an isomorphism of Fr\'echet spaces and for each compact $K\subset \bbR^n$, there is a constant $C_K$ such that $\|f\|^{r,hr}_K\leq \|F\|^r_{P(K)}$. 

\medskip

Conclusion: the space of invariant functions given by a function of class $C^r$ of the invariants is a closed subspace of $\calE^{r,hr}(\bbR^n)$, isomorphic to the Fr\'echet space~$\calE^r(P(\bbR^n)$ .

\section{Invariant functions of class $C^{hr}$ and functions of the invariants}

There exists a linear and continuous mapping (see [1]):
\[ 
L: C^{hr}(\bbR^n)^W \longrightarrow \mathcal E^r\big(P(\bbR^n)\big)\quad L(f)=F, 
\hskip 2mm f=F\circ P .
\] 
Thanks to the Whitney extension theorem, the target might be $C^r(\bbR^n)$ as well. 
The map $L$ is a version of $P^{*(-1)}$, we gave it another name because the source and target are different from those of $P^{*(-1)}$ in the previous section.

Clearly $L$ is injective and surjective on its image.

Let $C$ be a Weyl chamber. Its image $P(C)$ is the interior of $P(\bbR^n)$ where F is of class $C^{hr}$. The loss of differentiability happens on the border of $P(\bbR^n)$.

 Let $S$ be a stratum of $\barC  $, intersection of $p$ reflecting hyperplanes. On $P(S)$ the loss of differentiability is governed by $W_S$, the isotropy subgroup of $S$, and $F$ is of class $C^{hr/h_S}$, where $h_S$ is the highest degree of the $W_S$-invariants. Moreover, we have: 

\begin{prop} \label{Yrefprop}
If $f=F\circ P$ is of class $C^{hr}$, the derivatives of $F$ of order $\beta=(\beta_1, \ldots, \beta_n)$  with $\sum_{i=1}^n \beta_i k_i\leq hr$, are continuous on $P(\bbR^n)$. 
\end{prop}

\begin{proof}
Let us begin with two remarks. First it is sufficient to study the problem for each $W^i$ acting on $\bbR^{n_i}$, so we may and will assume $W$ to be irreducible. Then since $P$ is proper if $\partial^{|\beta|}F /\partial x^{\beta} \circ P$ is continuous, \hbox{so is $\partial^{|\beta|}F /\partial x^{\beta}$}.

The first derivatives of $F$ with respect to $p_j$ are solutions of the system:
\[
\Big(\frac{\partial f}{\partial x}\Big)
=
\Big(\Big(\frac{\partial p_i}{\partial x_j}\Big)_{1\leq i,j\leq n}\Big)
\Big(\frac{\partial F}{\partial p}\circ P\Big).
\]
The solutions are given by:
\[
c\Big(\prod_{\tau\in \mathcal R}\lambda_{\tau}\Big)
\frac{\partial F}{\partial p_j}\circ P 
= 
\sum_{i=1}^n(-1)^{i+j} 
M_{i,j}\frac{\partial f}{\partial x_i}\raise 2pt\hbox{, }\quad j=1,\ldots,n.
\]
where $M_{i,j}$ is the Jacobian of the polynomial mapping:
\[(z_1,..., z_{i-1},z_{i+1}, ...,z_n; z_i) \mapsto  (p_1 (z),..., p_{j-1}(z), p_{j+1}(z), ..., p_n(z) ; z_i)\]
This mapping is invariant by the subgroup $W_i$ of $W$, that leaves invariant the $i^{th}$ coordinate axis. $W_i$ is a reflection group, product of the identity on $\bbR e_i$ and the reflection group generated by the set $\mathcal R_i\subset \mathcal R$ of those of the hyperplanes generating~$W$ containing $\bbR e_i$.  

Actually $M_{i,j}$ is an homogeneous polynomial of degree $s_j=\sum_{i=1,i\neq j}^n (k_i-1) $ which is $s_j-1$ flat at the origin and divisible by $\prod_{\tau\in \mathcal R_i}\lambda_{\tau}$.

At the origin there is a loss of differentiability of $1+d-s_j=k_j$ units. In the neighborhood of a stratum $S$ intersection of $p$ reflection hyperplanes, we write 
$$
P=Q\circ V
$$
with $Q$ invertible and $V$ a Chevalley mapping for the isotropy subgroup~$W_S$ of~$S$. We have 
$$
M_{i,j}=\sum_{l=1}^n V_{l,j}Q_{i,l}.
$$
The $V_{l,j}$ and accordingly the $M_{i,j}$ are $(s_p^j-1)$-flat on $S$, where $s_p^j$ is the degree of the $V_{l,j}$. Since $p$ is the number of reflections in $W_S$, there is a loss of differentiability of $1+p-s_p^j$ equal to the degree $k'_j$ of the~$j^{\mathrm {th}}$ polynomial of $V$ which is less than or equal to $k_j$. So overall the derivation of $F$ with respect to $p_j$ brings a loss of differentiability of $k_j$ units. 

As a consequence a derivation of order $\beta$ brings a loss of differentiability of $\sum_{i=1}^n \beta_i k_i$ units and if this sum is less than $hr$, the derivative will be continuous on~$P(\bbR^n)$. 
\end{proof}

In the neighborhood of $S$, the loss of differentiability entailed by a $\beta$-derivation is $\sum_{i=1}^n \beta_ik'_i$ where the $k'_i$ are the degrees of the invariant polynomials of $W_S$. Since this sum is less than $\sum_{i=1}^n \beta_i k_i$, more derivatives will be continuous in the neighborhood of $S$.

 If $\sum_{i=1}^n \beta_i k_i=hr$, the derivative 
 $\partial^{|\beta|}F /\partial x^{\beta}$ is continuous at the origin. Its derivative with respect to $p_j$ is not continuous but its product by~$|x|^{k_j}$ tends to $0$ with x. 
 
 When computing $\partial^{|\alpha|}(F\circ P)/\partial x^{\alpha}, $ with $|\alpha|\leq hr$ by the Faa di Bruno formula, the derivative $\partial^{|\beta|}F /\partial p^{\beta} \circ P$ is multiplied by (we use the notations of [7]) :
\[
\varepsilon_{\beta}
=
\alpha ! \sum_{E_\beta}
\prod^n_{j=1} 
\prod_{A_\alpha} 
\frac{1}{\mu_{iq^i}}
\Big(\frac{1}{q^i!}
\frac{\partial^{|q^i|} p_j}{\partial x^{q^i}}
\Big)^{\mu_{iq^i}}
\] 
with 
\[
E_\beta=\Big\{
(\mu_{1q^1}, \ldots, \mu_{nq^n}) \pv \mu_{iq^i} \in \mathbb{N}, 1\leq |q^i|\leq s, \sum_{|q^i|=1}^s \mu_{iq^i} =\beta_i, 1\leq i\leq n 
\Big\}
\]  
and 
\[A_\alpha =\Big\{Q= (q^i_j)  \pv 1\leq i,j\leq n, 1\leq |q^i|\leq s, 1\leq i\leq n, \sum_1^n\sum_{|q_i|}^s\mu_{iq^i}q^i =\alpha\Big\}.
\]

\begin{rema}\rm
The $p_j$ are homogeneous polynomials of degree $k_j$. Up to a multiplicative constant,  the terms  
$\prod^n_{j=1} \prod_{A_\alpha}(\partial^{|q_i|} p_j/\partial x^{q_i})^{\mu_{iq^i}}$
are equal to $|x|^{\sum k_j\beta_j-|\alpha|}$.
This exponent is strictly positive when $\sum\beta_jk_j>hr$ and we have  
\[
|x|^{\sum k_j\beta_j-|\alpha|}
\Big|\frac{\partial^{|\beta|}F}{\partial p^{\beta}}\circ P(x)\Big|
\longrightarrow 0,\quad \hbox{when }  x\rightarrow 0.
\]
In the neighborhood of a stratum $S$, we could show that $\varepsilon_\beta(x)$ is equivalent to a power of the distance from $x$ to $S$.
\end{rema}

So, when $f\in C^{hr}(\bbR^n)^W$, the function $L(f)=F$ is in a subspace of $\mathcal E^r(P(\bbR^n))$, the functions of which are of class $C^{rh}$ on the interior of $P(\bbR^n)$, in $\mathcal E^{rh/h_S}(P(S))$ on the strata $P(S)$ of the border of $P(\bbR^n)$. Additionally the behavior of the derivatives of $F$ that are lost on the border is as follows.
 
When $\sum_{i=1}^n \beta_i k_i> hr$, if the derivative $(\partial^{|\beta|}F/\partial p^{\beta} )$ is lost on strata  $S=P(A)$ of the border of $P(\bbR^n)$, the function equal to  $\varepsilon _{\beta}\cdot(\partial^{|\beta|}F/\partial p^{\beta} )\circ P$ on the interior~$C$ of $P(\bbR^n)$ and equal to $0$ on $A$ is continuous.

To be more consistent, we put $\varepsilon_{\beta}=1 $ when $\sum_{i=1}^n \beta_i k_i\leq hr$ and for any $\beta$ we will consider the functions $\varepsilon _{\beta}\cdot\partial^{|\beta|} F/\partial p^{\beta} \circ P$ that vanish on the strata $A$ when $\partial^{|\beta|}F /\partial p^{\beta}$ does not exist on $P(A)$.

If we define the semi-norms 
\[
|F |^{r,hr}_{P(K_s)}= \max_{|\beta|\leq hr,\,x\in K_s}
\Big|\varepsilon_{\beta} (x)
\frac{\partial^{|\beta|}F}{\partial p^{\beta}}\circ P(x)\Big|,
\]
\[
\tttv F\tttv ^{hr}_{P(K_s)}= \|F\|^r_{P(K_s)} + |F|^{r,hr}_{P(K_s)} ,
\] 
the Faa di Bruno formula shows that $|f|^{hr}_{K_s}$ is equivalent to $|F |^{r,hr}_{P(K_s)}$. By Whitney regularity $\|f\|^{hr}_{K_s}$ and $| F |^{r,hr}_{P(K_s)}$ are also equivalent.

From proposition 1.1, there is a constant $B_{K_s}$ such that 
$\|F\|^r_{P(K_s)}\leq B_{K_s}\|f\|^{hr}_{K_s}$. 

So, we have $\tttv F\tttv ^{hr}_{P(K_s)}\leq B^1_{K_s}\|f\|^{hr}_{K_s}$ for some constant $B^1_{Ks}$.
 Hence the continuity of $L$ when the target has the topology induced by the semi-norms 
$\tttv .\tttv ^{hr}_{P(K_s)}$. 

Since $L$ is a bijection on its image, by the Banach theorem, we have:

\begin{prop}
The map $L$ is a Fr\'echet isomorphism of $C^{hr}(\bbR^n)^W$ onto its image in  $\mathcal E^r(P(\bbR^n))$ provided with the topology induced by the semi-norms  $\tttv F\tttv ^{hr}_{P(K_s)}$.
\end{prop}

\section* {References}

\parskip 3pt
\parindent -24pt
\leftskip -\parindent

[1]  G. Barban\c con. ---  Chevalley theorem in class $C^r$, 
{\it Proc.\kern 2pt Edinb. Math. Soc.} 
\textbf{139 A} (2009), pp.\kern 2pt  743--758.

[2]  G. Barban\c con. --- Whitney regularity of the image of the Chevalley mapping, 
{\it Proc.\kern 2pt Edinb. Math. Soc.} \textbf{146 A} (2016), pp.\kern 2pt  895--904

[3]  C. Chevalley.  --- Invariants of a finite group generated by reflections.  
{\it Amer. J.  Math.} \textbf{77}  (1955), pp.\kern 2pt  778--782.

[4]  H.S.M. Coxeter.  --- The product of the generators of a finite group generated by reflections.  {\it Duke Math. J.} 
\textbf{18}  (1951), pp.\kern 2pt   765--782.

[5]  Hernandez Encinas \& Munoz Masque. --- A short proof of the generalized Faa di Bruno's formula. {\it Applied Math. Letters} \textbf{16}  (2003), pp.\kern 2pt    975--979. 

[6]  G.\kern 2pt  Glaeser.  --- \'Etude de quelques alg\`ebres tayloriennes.  {\it J. Anal. Math.}
\textbf {6}  (1958), pp.\kern 2pt    1--124.

[7]  G.\kern 2pt  Glaeser. --- Fonctions compos\'ees diff\'erentiables. 
{\it Ann. of Math.} \textbf {77}  (1963), pp.\kern 2pt    193--209. 

[8]  M.L.\kern 2pt  Mehta.  --- Basic set of invariant polynomials for finite reflection groups.
{\it Commun. Alg.} \textbf{16}  (1988), pp.\kern 2pt    1083--1098.

[9]  J.C.\kern 2pt  Tougeron.  --- Id\'eaux de fonctions diff\'erentiables. 
{\it Ergeb. Math. Grenzgeb.}  \textbf{71},  Springer (1972).

[10]  H.\kern 2pt  Whitney.  --- Functions differentiable on the boundary of regions. 
{\it Ann. of Math.}  \textbf{35}   (1934), pp.\kern 2pt    482--485.

[11]  H.\kern 2pt Whitney.  ---
Analytic extension of differentiable functions defined in closed sets. 
{\it Ann. of Math.}  \textbf{36}    (1934), pp.\kern 2pt    63--89.
\end{document}